\documentclass[8pt, a4paper, twocolumn]{article}
\usepackage{titlesec}
\titleformat*{\section}{\normalsize\bfseries\sffamily}

\usepackage[utf8]{inputenc}
\usepackage{geometry}
\usepackage{subfiles}
\usepackage{enumitem}
\usepackage{comment}

\usepackage{amsmath}
\usepackage{amsfonts}
\usepackage{amssymb}
\usepackage{amsthm}
\usepackage{mathtools}

\usepackage{caption}
\usepackage{subcaption}

\newcommand{\citep}{\cite}
\newcommand{\citet}{\cite}
\newcommand{\citeyear}{\cite}

\newtheorem{theorem}{Theorem}
\newtheorem*{theorem*}{Theorem}
\newtheorem{proposition}{Proposition}

\newtheorem*{lemma*}{Lemma}
\newtheorem*{remark*}{Remark}

\theoremstyle{definition}

\newtheorem{definition}{Definition}

\newtheorem*{acknowledgement}{Acknowledgement}

\usepackage{natbib}
\begin{document}

\title{\vspace{-1cm} The Shift Operator Calculus for Stationary Time Series Analysis}
\author{
  Anand Ganesh \thanks{National Institute of Advanced Studies, Manipal Academy of Higher Education, email: anandg@nias.res.in} \and
  Babhrubahan Bose \thanks{Department of Mathematics, Indian Institute of Science, email: babhrubahanbose@gmail.com} \and
  Anand Rajagopalan \thanks{Independent ML Researcher, email: anandbr@gmail.com}
}
\date{}

\maketitle
\begin{abstract}
The article establishes a rigorous shift operator calculus for stationary time series modeling, addressing a certain gap in the literature. It provides proofs of existence and isometry for the transfer function operators $f(B)$ and $f(T)$ where $B$ is the bilateral shift operator and $T$ is the unilateral shift operator for different families of functions $f$. The article establishes convergence of the power series of $f(B)$ and $f(T)$ under the operator norm for the Wiener algebra $\mathbb{W}_+$, and convergence under strong operator topology for $f$ in $H^{\infty}$, based on the use of Abel sums. Based on this calculus, it unifies the notion of stationary process invertibility with the operator invertibility of the transfer function $f(T)$. 
\end{abstract}


\section{Introduction}

Let us consider a purely linearly non-deterministic stationary process $X_t$ and its moving average representation $X_t = f(B) \epsilon_t$ where $B$ is the bilateral shift operator, a unitary operator also known as the lag operator or backshift operator, and $\epsilon_t$ represents white noise. We note a gap in time series literature in the use of operator valued transfer functions like $f(B)$ since it does not specify the nature of functions $f$ for which $f(B)$ is well-defined, or the nature of convergence for the power series of $f(B)$. The current article addresses this gap by first establishing norm convergence of the power series of $f(B)$ for $f \in \mathbb{W}_+$ (Proposition \ref{existence_theorem_w}), matching the $\ell^1$ criterion commonly used in time series analysis. We then invoke the operator theoretic analysis in \cite{NagyFoias2010} to relax this $\ell^1$ criterion and include a larger class of $H^\infty$ functions using Abel sums, with convergence under the strong operator topology (Proposition \ref{unitary_existence_theorems}). The most general representation of linearly non-deterministic stationary processes $X_t$ based on the operator notation $f(B)$ requires the use of unbounded operators with $f \in H^2$, but further work is required to understand these.

We next identify a peculiar mismatch between operator theoretic literature and applied time series literature in the use of the bilateral shift operator $B$, as seen in the above moving average representation $X_t = f(B) \epsilon_t$. It is a well known result of von Neumann that an isometric operator $V$ can be decomposed into a unitary operator $B$ and a pure isometry $T$, and that this decomposition corresponds respectively to the decomposition of a stationary process $Z_t$ into a linearly deterministic part $Y_t$ and a purely linearly non-deterministic part $X_t$ respectively. The mildly paradoxical situation is that time series literature traditionally uses a {\em unitary} operator, namely the bilateral shift operator $B$, for the {\em non}-unitary part $X_t$ of this decomposition, which represents the non-deterministic part of the process. The current article clarifies and resolves this paradoxical situation by introducing the unilateral shift operator $T$ with proofs of existence and isometry for $f(T)$ (Proposition \ref{hinfinity_isometry_theorem}). We believe the upcoming discussions on invertibility and mean ergodicity illustrate the value of a rigorous functional calculus for $f(B)$ and $f(T)$.

The use of $f(T)$, being in line with the von Neumann decomposition, naturally unifies the invertibility condition for a stationary process $X_t$ based on the roots of its transfer function $f(z)$, with the operator theoretic invertibility of its transfer function operator $f(T)$ (Proposition \ref{invertibility_equivalence_theorem}). There are numerous definitions for the invertibility of a stationary process $X_t$, some based on the roots of a polynomial (\cite{Harvey1990}), some based on an $\ell^1$ condition on the MA representation (\cite{BrockwellDavis1991}) and some based on the convergence of the AR polynomial on the unit circle, as in (\cite{BoxJenkins1976}). Proposition \ref{invertibility_equivalence_theorem} shows that the operator invertibility of $f(T)$ is consistent with these earlier definitions, and generalizes their applicability from the Wiener algebra $\mathbb{W}_+$ (corresponding to the $\ell^1$ condition) to the larger domain of $H^\infty$ functions. Further, the operator theoretic definition provides a clean algebraic picture within the world of stationary processes allowing us to place some heuristic connections between stationarity and invertibility, referred to as a {\em sort of} duality by \cite{Harvey1990}), in a more rigorous framework. Finally, for $f \in H^\infty$ we identify $f(T)$ with the Toeplitz operator $T_f$ (Proposition \ref{toeplitz_operator_theorem}), and this allows us to invoke the invertibility properties of $T_f$ to understand the invertibility of $f(T)$ and $X_t$.

The rest of the paper is organized as follows. We first recall some key definitions in Section \ref{section_preliminaries}. In Section \ref{section_operator_calculus} we substantiate the claimed theoretical gap based on a brief literature survey, and then establish a rigorous operator calculus for $f(B)$ and $f(T)$. We discuss invertibility including some historical background in Section \ref{section_invertibility} and mean ergodicity in Section \ref{section_mean_ergodicity}, and this is followed by a brief conclusion.

\section{Preliminaries} \label{section_preliminaries}

In this paper we work in the setting of a separable Hilbert space $\mathcal{H}$ (a complete inner product space with a countable basis) and its bounded linear operators $B(\mathcal{H})$. We recall the notion of an adjoint $T^*$ for an operator $T \in B(\mathcal{H})$ as the unique operator in $B(\mathcal{H})$ satisfying the adjoint condition $\langle Tx, y \rangle = \langle x, T^*y \rangle$.

\begin{definition}[Bilateral Shift Operator]
Let an orthnormal basis for $\mathcal{H}$ be given by $\{e_n:n\in\mathbb{Z}\}$. Then the bilateral shift operator with respect to this orthonormal basis is defined to be the bounded linear operator $B \in\mathcal{B}(\mathcal{H})$ satisfying $B(e_n) = e_{n+1}$.
\end{definition}

\begin{definition}[Unilateral Shift Operator]
Let an orthnormal basis for $\mathcal{H}$ be given by $\{e_n:n\in\mathbb{N}\}$. Then the unilateral shift operator with respect to this orthonormal basis is defined to be the bounded linear operator $T\in\mathcal{B}(\mathcal{H})$ satisfying $T(e_n) = e_{n+1}$.
\end{definition}

We now recall the following theorem of Nagy \cite[p.155]{Agler2002} which is key to connecting the bilateral shift operator $B$ to the unilateral shift operator $T$:
\begin{theorem}[Nagy dilation]
Every contraction $V$ in a Hilbert space $\mathcal{H}$ has a unitary dilation $U$ (in some superspace $\mathcal{H}'$ containing $\mathcal{H}$).
\end{theorem}

The bilateral shift operator $B$ is in fact the minimal unitary dilation of the unilateral shift operator. In this context, we use the notation pr as in \cite[p.9]{NagyFoias2010} to indicate a projection from the superspace $\mathcal{H}'$ to the subspace $\mathcal{H}$. For instance, $T = \text{pr~}B$ indicates that $Tx = P_{\mathcal{H}}Ux$ where $P_{\mathcal{H}}$ is a projection operator that projects onto $\mathcal{H}$ and $x \in \mathcal{H}$.

We assume familiarity with the definitions of the function spaces $L^p$ and and sequence spaces $\ell^p$ for $1 \le p \le \infty$. We define the Hardy space $H^p$ as follows. We consider holomorphic functions $g : \mathbb{D} \to \mathbb{C}$ such that their radial limit on $S^1$ exists and is in $L^p(S^1)$. We identify $H^p$ with these limit functions from $S^1 \to \mathbb{C}$. For $f \in L^{\infty}$, we define the {\em Toeplitz operator} $T_f: H^2 \to H^2$ by $T_fg(z) := P(f(z)g(z))$ where $g \in H^2$ and $P$ represents the projection operator from $L^2$ to $H^2$.

\begin{definition}[Wiener Algebra]
The Wiener algebra $\mathbb{W}$ consists of functions on the unit circle $f:S^1 \to \mathbb{C}$ representable as a power series $f(z) = \sum_{n=-\infty}^{\infty} a_n z^n$ with $(a_n) \in \ell^1$. We use the notation $\mathbb{W}_+$ to denote the subspace of these functions with $a_n = 0$ for $n < 0$.
\end{definition}

We next define a stationary stochastic process or sequence starting with a probability space $(\Omega, \mathcal{F}, \mathbb{P})$. Let $L^2(\Omega)$ be the Hilbert space of random variables $X$ on $\Omega$ having finite second moment.

\begin{definition}[Stationary Process]
A weakly stationary sequence (resp. bi-sequence) of random variables $\{ X_n \}_{n=0}^{\infty}$ (resp. $\{ X_n \}_{n=-\infty}^{\infty}$) on $\Omega$ is defined to be a sequence (resp. bi-sequence) in $L^2({\Omega})$ such that cov$(X_n, X_{n+k})$ only depends on $k$ for all $k \in \mathbb{Z}$. We will omit the word weakly henceforth, use the term stationary process to refer to both sequences and bi-sequences of such random variables.
\end{definition}

We now recall the famous decomposition theorem of Wold as stated in \citep[p.109]{Hamilton1994}.

\begin{theorem}
{\em Wold Decomposition}: Any zero-mean covariance-stationary process $X_t$ can be represented in the form $X_t = \sum_{j=0}^{\infty} x_j \epsilon_{t-j} + \kappa_t.$

\end{theorem}

Here $\epsilon_t$ is white noise representing {\it linear forecasting error}, also called {\it innovation}, $x_0 = 1$ and $\sum_{j=0}^{\infty} x_j^2 < \infty$ (i.e. $\{x_j\} \in \ell^2$). $\sum_{j=0}^{\infty} x_j \epsilon_{t-j}$ is the linearly non-deterministic component while $\kappa_t$ is called the lineary deterministic component. We focus on the case $\kappa_t = 0$ in which case $X_t$ is called a purely linearly non-deterministic process. In the interest of brevity, we use the term stationary process in the present article to refer exclusively to these purely linearly non-deterministic stationary processes.

A stationary process represented as a linear combination of white noise processes $\epsilon_t$, as in the Wold decomposition, is called a moving average ($MA$) process. As per the theorem above, a purely linearly non-deterministic stationary process always has a moving average representation. A representation of a stationary process $X_t$ based on a linear combination of past random variables $X_{t-j}$ for $j > 0$ is called an autoregressive ($AR$) representation. An $AR$ represenation is not always guaranteed, and the following notion of invertibility tells us when an $AR$ representation is possible for a given stationary process. We list multiple definitions of invertibility starting with the simpler one.

\begin{definition}[Mann-Wald Stationarity]
Given an AR process with $\epsilon_t = \sum_{n=0}^{N} a_n X_{t-n}$, the roots of the polynomial $p(z) = \sum_{n=0}^{N} a_n z^{N-n}$ should lie {\em inside} the unit circle $S^1$.
\end{definition}

The upcoming definition of invertibility in \cite[p.65]{Harvey1990} looks somewhat similar to the above definition of stationarity, but involves two changes. The first is that it starts with a MA representation rather than an AR representation. The other change is that the polynomial considered has reversed order which leads to roots being {\em outside} the circle, and this is the standard convention at this time. The reversal makes no substantive difference, and as we will see later in the text.

\begin{definition}[Harvey Invertibility]
A moving average process $X_t = \sum_{n=0}^{N} a_n \epsilon_{t-n}$ is said to be invertible when all the roots of the polynomial $f(z) = \sum_{n=0}^{N} a_n z^n$ lie {\em outside} the unit circle $S^1$.
\end{definition}

The Harvey definition can only be applied when the transfer function has a finite number of terms. This limitation is removed in the following generalized definition.

\begin{definition}[$\ell^1$ Invertibility]
\cite[p.86]{BrockwellDavis1991} A stationary process $X_t$ is said to be invertible if we can rewrite the same as an autoregressive $AR(\infty)$ process as $X_t = \epsilon_t + \sum_{n=1}^{\infty} b_n X_{t-n}$ with $\sum_{n=0}^{\infty} |b_n| < \infty$.
\end{definition}

The condition $\sum_{n=0}^{\infty} |b_n| < \infty$ is referred to generically as the $\ell^1$ condition in this article, and corresponds to the constraint $f \in \mathbb{W}_+$ on the transfer function $f$. The essential part of this definition is that we can rewrite an invertible stationary process as an autoregressive process, and the $\ell^1$ condition can be further relaxed as we show later in the text.

\section{Operator Calculus} \label{section_operator_calculus}

Time series literature makes use of the notation $f(B)$ referring to operator valued transfer functions, but it does not specific the nature of functions $f$ for which $f(B)$ is well-defined. The notation $f(B)$ is first observed in \cite[p.127, Eq.(2.10)]{Whittle1953}, and then subsequently in numerous texts like \cite[p.14]{BoxJenkins1976}, \cite[p.25]{Hamilton1994} or \cite[p.86]{Shumway2006}, thus becoming standard over the years. \cite[p.83]{BrockwellDavis1991} address this lack of rigor indirectly by focusing on random variable convergence rather than operator definition per se, and end up prescribing the somewhat restrictive $\ell^1$ condition. We also note that the rigorous treatments in \cite{Grenander1957, Rozanov1963} and \cite{Hannan1970}, avoid the use of the operator formalism $f(B)$. While they introduce the bilateral shift operator $B$, its powers $B^n$ and their finite linear combinations, they studiously avoid the use of their infinite sums and the suggestive notation $f(B)$. In the current article, we take a rigorous look at these infinite sums starting with $f \in \mathbb{W}_+$, in line with the above $\ell^1$ condition. We prove that $f(B)$ and $f(T)$ are well-defined in this case, and later generalize this result to the larger family of $H^{\infty}$ functions.

Let $f \in \mathbb{W}_+$ with $f(z) = \sum_{n=0}^{\infty} a_n z^n$ restricted to positive powers of $z$, represent a transfer function. Consider the following power series expressions:
\begin{align*}
    f(B) = \sum_{n=0}^{\infty} a_n B^n \\
    f(T) = \sum_{n=0}^{\infty} a_n T^n.
\end{align*}

We claim that the power series expressions for $f(B)$ and $f(T)$ are well defined.

\begin{proposition} \label{existence_theorem_w}
For $f \in \mathbb{W}_+$, $f(B)$ is a well defined bounded linear operator on the domain of definition of $B$, and the power series converges under the operator norm.
\end{proposition}
\begin{proof}
We prove that the norm of $f(B)$ is bounded by looking at its power series expansion. This is a simple Banach algebra argument based on the norm of $B$.
\begin{gather*}
\|f(B)\| = \| \sum a_n B^n \| \le \sum |a_n B^n| = \sum |a_n| = \| f \|_{\mathbb{W}}.
\end{gather*}
\end{proof}

Our existence results for $H^\infty$ follow directly from \cite[p.112, Chapter III, Section 2]{NagyFoias2010}. The required results are scattered across the text, so we summarize them below, starting with the following Abel sum definition.

\begin{definition}[Abel Sum]
Given $f \in H^\infty$ as before we have the following definitions of $f_a(B)$ and $f_r(B)$.
\[ f_a(B) = \lim_{r \to 1-0} f_r(B) = \lim_{r \to 1-0} \sum_{n=0}^{\infty} r^n a_n B^n \]
\end{definition}

Note that this definition is a departure from the spectral definition of $f(B)$ traditionally used in time series analysis. We now summarize their key results as follows:

\begin{theorem}[Nagy-Foias] \label{nagy_foias_theorem}
For an isometric operator $V$, a unitary operator $U$, a function $f \in H^{\infty}$, and a sequence of functions $f_n$ also in $H^{\infty}$ we have the following:
\begin{enumerate}[label=(\alph*),nosep]
	\item $f_r(U)$ and $f_r(V)$ converge strongly to the Abel sums $f_a(U)$ and $f_a(V)$ respectively. \label{enum_abel_convergence}
	\item If $f_n(z)$ converges to $f(z)$ uniformly on the open unit disk $\mathbb{D}$, then $f_n(U)$ converges to $f_a(U)$ under the operator norm. \label{enum_uniform_convergence}
	\item If $f_n$ converges to $f_a$ boundedly on the unit circle $S^1$, $f_n(U)$ converges to $f_a(U)$ under the strong operator topology. \label{enum_bounded_convergence}
	\item $\|f_a(V)\| \le \|f\|_\infty$. \label{enum_norm_inequality}
	\item When $U$ is the minimal unitary dilation of $V$, we also have $f_a(V) = {\text pr} f_a(U)$. \label{enum_projection}
\end{enumerate}
\end{theorem}

In the above result, $f_n$ refers to a general sequence of functions, not necessarily partial sums of $f$. In fact, for $f \in H^\infty$, the partials sums of $f$ do not, in general, converge boundedly to $f(z)$. Thus, in-lieu of partial sum convergence as in \ref{enum_uniform_convergence} or \ref{enum_bounded_convergence}, we have to be satisfied with Abel sum convergence as in \ref{enum_abel_convergence}. 

\begin{proposition}[Bilateral Functional Calculus] \label{unitary_existence_theorems}
For $f \in H^\infty$, $f_a(B)$ and $f_a(T)$ are well defined as the strong operator limit of $f_r(B)$ and $f_r(T)$ respectively, and $f_b(T) = \text{pr}~f_a(B)$.
\end{proposition}

\begin{proof}
$B$ is a unitary operator and $T$ is an isometry, and further $B$ is the minimal unitary dilation of $T$. The proposition now follows immediately from \ref{enum_abel_convergence} and \ref{enum_projection} above.
\end{proof}

Note that we don't have power series convergence in general for $f \in H^{\infty}$, unless the partial sums of $f$ happen to converge boundedly, but we still have a well defined operator based on Abel sums. In fact, the power series of $f$ must be regarded in general as a symbolic shorthand for the Abel sum, though operator norm convergence holds for $f \in \mathbb{W}_+$.

We will now prove an isometry result, that $\|f(T)\| = \|f\|_{\infty}$. For the specific case of the unilateral shift operator $T$, this result is an improvement on the norm inequality \ref{enum_norm_inequality} above. Further, it is known that the Szeg\"o kernel is the eigenvector of the adjoint of the multiplication operator \cite[p.10]{Agler2002}. The proof below contains an interesting variant of this idea.

\begin{proposition} [Isometry] \label{hinfinity_isometry_theorem}
For $f \in H^\infty$, $\|f(T)\| = \|f\|_{\infty}$.
\end{proposition}
\begin{proof}
For $f \in H^\infty$, let $M_{f(z)}$ denote the multiplication operator on $H^2(\mathbb{D})$ such that $M_{f(z)}g(z) = f(z)g(z)$ for $g \in H^2(\mathbb{D})$. The proof notes that $f(T)$ is essentially the same as this multiplication operator $M_{f(z)}$, and $T$ is essentially $M_z$; and for reproducing kernel Hilbert spaces we can show that $\| M_{f(z)}\| \ge \| f \|_{\infty}$. The following is a slight elaboration of these ideas.

Consider the map $U:H^2(\mathbb{D}) \to \mathcal{H}$ defined by $U(z^n) := e_{n+1}$ for $n\geq0$. We can see that $M_z = U^*TU$. Now, $U$ maps an orthonormal basis to another and hence is a unitary linear map. By continuity and linearity of the operator $U$, we get
\begin{align*}
	f(T) &= \sum\limits_{n=0}^\infty a_nT^n = \sum\limits_{n=0}^\infty a_n U M_z^n U^*=U\left(\sum\limits_{n=0}^\infty a_nM_z^n\right)U^* \\
	&= U M_{\sum\limits_{n=0}^\infty a_nz^n}U^* = U M_{f(z)}U^*.
\end{align*}
Thus, $M_{f(z)}$ is unitarily related to $f(T)$, and $\| M_{f(z)} \| = \| f(T) \|$. But $\|f(T)\|\leq\|f\|_{\infty}$, and thus it suffices to show that $\left\|M_{f(z)}\right\| \geq \|f\|_\infty$.

Now, $H^2(\mathbb{D})$ is a reproducing kernel Hilbert space with respect to the Szeg\"o kernel given by:
\[ K_w(z) = \frac{1}{1-\overline{w}z}. \]
Based on the reproducing property of $K_w(z)$ we have:
\[\left\langle M_{f(z)}K_w(z), K_w(z)\right\rangle = \left\langle f(z)K_w(z), K_w(z)\right\rangle = f(w)K_w(w).\]
Further, $K_w(w) = \left\langle K_w(z), K_w(z)\right\rangle = \left\|K_w\right\|^2>0$. Thus,
\[\left\|M_{f(z)}\right\| \geq \left|\frac{\left\langle M_{f(z)}K_w(z), K_w(z)\right\rangle}{\left\langle K_w(z), K_w(z)\right\rangle}\right| = |f(w)|. \]
Since $w\in\mathbb{D}$ is arbitrary in the above equation, we obtain
\[\left\|M_{f(z)}\right\| \geq\|f\|_\infty.\].
\end{proof}

We now provide a new alternate proof of power series convergence for $f \in \mathbb{W}_+$. This proof works directly with the spectral definition, and does not make use of  Abel sums. It provides a little more insight into the boundedness requirement, and its connection to the dominated convergence theorem. It also highlights the need for the Abel sum definition since the spectral definition does not extend further to $H^\infty$ or even to the disk algebra $\mathbb{A}(\mathbb{D})$.

\begin{proposition} \label{existence_theorem_a}
(Existence) For $f \in \mathbb{W}_+$, $f(B)$ is a well defined bounded linear operators on the domain of definition of $B$, and the power series converges under the strong operator topology.
\end{proposition}
\begin{proof}
We invoke the spectral theorem as follows (\cite[p.271, Section 4]{Conway1997}):
\begin{align*}
f_s(B) &= \int_{S^1} f(\lambda) dE_{\lambda} \\
     &= \int_{S^1} \sum_{n=0}^{\infty} a_n \lambda^n dE_{\lambda}.
\end{align*}

The above holds since $\sum_{n=0}^{\infty} a_n \lambda^n$ converges to $f(\lambda)$ everywhere on $S^1$. Further, for $f \in \mathbb{W}_+$, the partials sums $f_N(\lambda) = |\sum_{n=0}^{N} a_n \lambda^n|$ converge boundedly to $f(\lambda)$ on the unit circle. Thus, we can invoke the dominated convergence theorem (which holds for projection valued measures \cite[p.214]{Bhatia2009}) as follows.

\begin{gather*}
f_s(B) = \sum_{n=0}^{\infty} a_n \int_{S^1} \lambda^n dE_{\lambda} = \sum_{n=0}^{\infty} a_n B^n = f(B).
\end{gather*}
\end{proof}

We note here that the proof does not go through, at least as-is, for $f \in H^{\infty}$. In particular, the dominated convergence theorem requires that the sequence of partial sums $f_n$ converge boundedly to $f$. But this is not true in general for $f \in H^{\infty}$ or even for $f \in \mathbb{A}(\mathbb{D})$, though it is in fact true for $f \in \mathbb{W}_+$. To handle $f \in H^{\infty}$ more generally we need the alternate definition $f_a(B)$ as described earlier. Finally, we note that for $f \in \mathbb{W}_+$, $f_a(U) = f_s(U)$ \cite[p.114, Eq.2.14]{NagyFoias2010}.

The above discussion extended the definition of $f(B)$ and $f(T)$ to $H^\infty$ functions based on the use of Abel sums, but this required a weaker form of convergence. It is possible to achieve operator norm convergence based on Cesaro sums, but we will not pursue this further.

\section{Stationary Process Invertibility} \label{section_invertibility}

We begin our discussion of invertibility with \cite[p.50]{BoxJenkins1976} who look at a simple MA(1) process and discuss conditions under which it admits a well defined AR representation. In particular they ask that a certain power series associated with the AR representation, essentially a geometric series denoted as $\pi(B)$, converge on or within the unit circle. Then they go on to define similar conditions for more general MA processes based on analogy with the MA(1) example, a sort of heuristic extrapolation of ideas. Their conditions mix complex function theory with operator notation $\pi(B)$, but in effect, they are seeking conditions under which the inverse operator $\pi(B)$ exists, which in turn guarantees a well defined AR representation. The heuristic arguments are insightful up to a point, but the descriptions are not complete, and one is left with questions like how is $\pi(B)$ even defined in the first place for an infinite order MA process where $\pi(B)$ is not easy to construct, and the convergence criterion is difficult to apply.

Subsequent treatments by \cite{BrockwellDavis1991} and others turn these heuristic but open-ended ideas into a clean $\ell^1$ criterion that is easier to manage. This $\ell^1$ criterion is elegant and simple, and mirrors the essential idea of \cite{BoxJenkins1976} that the inverse be well defined (that the power series converge on and within the unit circle). On the other hand, it silently confounds necessary conditions for the existence of $\pi(B)$ with sufficient conditions. What we need is simply that the inverse operator $\pi(B)$ exist, in other words, that it is well defined.

Our discussion of invertibility tries to argue from this broader viewpoint of necessary conditions, and we show that one can define invertibility in the larger space of $H^\infty$ functions as well. In particular, we show that operator invertibility of $f(T)$ for $f \in H^\infty$ is entirely consistent with prior treatments on ensuring an AR representation, but drops the sufficient but stronger $\ell^1$ condition in favor of the weaker but {\em necessary} requirement of operator invertibility. The definition in \cite{Harvey1990} based on the roots of the transfer function polynomial is limited to finite order MA representations, but it does not confound the necessary with the sufficient, and so, without loss of generality, we use the Harvey definition in our proof of Proposition \ref{invertibility_equivalence_theorem}.

In preparation for Proposition \ref{invertibility_equivalence_theorem} connecting stationary process invertibility with operator invertibility, we first discuss an example by way of motivation, in essence the same MA(1) example discussed in \cite{BoxJenkins1976}. We then list or prove a few preparatory theorems ahead of the main one.

By way of example, consider the stationary process $X_t = \epsilon_t -2 \epsilon_{t-1}$ associated with the transfer function $f(z) = 1 - 2z$. The transfer function has a root $z = \frac{1}{2}$ inside the unit circle, and is considered non-invertible for the purposes of stationary process modeling \cite[p.67]{Hamilton1994}. This matches the idea that $f(T) = 1 - 2T$ is non-invertible in the corresponding Wiener algebra. If instead $f(z)$ is represented using the unitary operator as $f(B) = 1 - 2B$, we find that $f(B)$ is in fact invertible which would not be faithful to the notion of ARMA model invertibility. For instance:
\begin{gather*}
    \frac{1}{1 - 2B} = -\frac{B^{-1}}{2} \frac{1}{1 - \frac{1}{2}B^{-1}} = -\frac{B^{-1}}{2}(1 + \frac{B^{-1}}{2} + \frac{B^{-2}}{4} + \dots).
\end{gather*}
Thus $f(B)$ has a perfectly valid inverse, while $f(T)$ does not have a similar inverse since $T^{-1}$ does not exist. That is, the operator invertibility of $f(T)$ matches the stationary process invertibility $X_t$. thus unifying the two notions of invertibility. We formalize these observations in Proposition \ref{invertibility_equivalence_theorem} below. 

We now observe that for $f \in \mathbb{W}$, $f(T)$ is the same as $T_f$, the Toeplitz operator. This result allows us to link the invertibility of $f(T)$ with known invertibility results on the Toeplitz operator, and we make use of this connection in the proof of Proposition \ref{invertibility_equivalence_theorem}. 

\begin{proposition} \label{toeplitz_operator_theorem}
	Given $f \in H^{\infty}$, $f(T) = T_f$.
\end{proposition}
\begin{proof}
	For $f \in H^\infty$ and $g \in H^2$ it is easy to see that product $f(z)g(z) \in H^2$. Thus $T_fg(z) = P(f(z)g(z)) = f(z)g(z)$, and thus $T_f$ is just the multiplication operator for $f \in H^\infty$. Now, letting $f(z) = a_0 + a_1z + a_2z^2 + \dots$ and $Tz^n = z^{n+1}$, let us consider the action of $f(T)$ and $T_f$ on the $z^n$:
	\[
	f(T) z^n = (\sum_{k=0}^{\infty} a_k T^k) z^n = \sum_{k=0}^{\infty} a_k z^{n+k}.
	\]
	Similarly,
	\[
	T_f z^n = (\sum_{k=0}^{\infty} a_k z^k) z^n = \sum_{k=0}^{\infty} a_k z^{n+k}.
	\]
	Since $T_fz^n = f(T)z^n$ for all basis elements $z^n \in H^2$, we have $f(T) = T_f$.
\end{proof}

We also make use of the following theorem \cite[p.139, Proposition 6.18]{Douglas1997} related to the invertibility of $H^\infty$ functions \footnote{The original statement uses a different convention for $H^\infty$, but what they refer to as $\hat{f}$ is essentially the same as the transfer function $f(z)$ in our convention.}.

\begin{theorem} \label{theorem_hinf_invertibility}
For $f \in H^\infty$, $f$ is invertible in $H^\infty$ if and only if $f|_\mathbb{D}$ is bounded away from zero.
\end{theorem}

We now have the tools necessary to state and prove our main result connecting stationary process invertibility with operator invertibility.

\begin{proposition} \label{invertibility_equivalence_theorem}
Given $f \in H^\infty$, and a stationary process represented as $X_t = f(T) \epsilon_t$, operator invertibility of $f(T)$ is consistent with, and subsumes the Harvey definition of invertibility for a stationary process $X_t$.
\end{proposition}

Note that the Harvey definition of invertibility for $X_t$ is applicable only when the MA representation is of finite length. In other words, when the corresponding transfer function is in fact a polynomial. This definition has been generalized to the Wiener algebra $\mathbb{W}_+$ as seen in the treatments of \cite{Hamilton1994} and \cite{Shumway2006}. We claim that the operator invertibility of $f(T)$ generalizes this further to $f \in H^\infty$.

\begin{proof}

The overall proof structure involves proving the following two subclaims. Without loss of generality, we will prove these two subclaims for the specific case when $f(T)$ is a polynomial for which Harvey invertibility is defined.
\begin{enumerate}[label=(\alph*),nosep]
	\item $X_t$ is invertible implies $f(T)$ is invertible.
	\item $f(T)$ is invertible implies $X_t$ is invertible.
\end{enumerate}

We will start with (a). Since $X_t$ is invertible, we have a polynomial $f(z) = \prod_{i=1}^{N} (x - \alpha_i)$ where $\alpha_i$ are outside the unit circle $S^1$. This implies $f(T) = \prod_{i=1}^{N} (T - \alpha_i) = \prod_{i=1}^{N} -\alpha_i(1 - \alpha_i^{-1}T)$. Now each of the $T - \alpha_i$ can be inverted using a geometric series as:
\[
	g_i(T) = -\alpha_i (1 - \alpha_i^{-1}T)^{-1} = -\alpha_i(1 + \alpha_i^{-1}T + \alpha_i^{-2}T^2 + \cdots).
\]

Further, each of these geometric series represents an operator function $g_i(T)$ with $g_i \in H^{\infty}$. Thus the product of $g_i$ continues to be in $H^{\infty}$, and converges under the strong operator topology. Thus $g(T) = \prod_{i = 1}^{N} g_i(T)$ is well defined, and thus $f(T)$ is invertible with inverse $g(T)$. This completes the proof of (a).

We will now prove (b). Given $f(T)$ invertible, and the equivalence of $f(T)$ and $T_f$ as in Proposition \ref{toeplitz_operator_theorem}, the Toeplitz operator $T_f$ is invertible. For $f \in H^\infty$, $T_f$ is invertible if and only if $f$ is invertible in $H^\infty$ (Winter's Theorem, \cite[p.164]{Douglas1997}). But this only happen when zero is not in the essential range of $f$ (Theorem \ref{theorem_hinf_invertibility}). For polynomials $f$, this means the roots of $f(z)$ have to be outside the unit circle. This completes the proof.

\end{proof}

The condition $f \in \mathbb{W}_+$ is an $\ell^1$ condition on the moving average coefficients of $X_t$. When $f(T)$ is invertible, it follows from the definition of the Wiener algebra, or from Wiener's $\frac{1}{f}$ invertibility theorem, that this $\ell^1$ condition applies to the AR coefficients as well. Such an $\ell^1$ condition on the AR coefficients is introduced in \cite[p.84, Proposition 3.1.2]{BrockwellDavis1991} in a slightly different manner based on the following theorem:
\begin{theorem}[$\ell^1$ Invertibility Condition] \label{brockwell_invertibility_theorem}
\cite[p.84, Proposition 3.1.2]{BrockwellDavis1991} If $\{X_t\}$ is a stationary process and we have a sequence $(a_n)$ such $\sum_{j=-\infty}^{\infty} |a_n| < \infty$, then the following series converges absolutely with probability one and in mean square to the same limit:
\[
    f(B) X_t = \sum_{n=-\infty}^{\infty} a_n X_{t-n}.
\]
\end{theorem}

Proposition \ref{invertibility_equivalence_theorem} shows that the $\ell^1$ condition can be relaxed, by requiring mean square convergence instead of absolute convergence in probability. \cite{Buhlmann1995} and \cite{Berk1974} connect this $\ell^1$ condition to Wiener's $\frac{1}{f}$ invertibility condition. But their approach is Fourier analytic, and unlike the current paper, there is no attempt at identifying the operator algebra structure beneath the $\ell^1$ condition, in particular the Wiener algebra structure of the transfer functions $f(z)$ or $f(B)$. 

\subsection{Historical Commentary}

Having shown the use of the unilateral shift operator $T$ in analyzing invertibility, the current section tries to restore the balance by reviewing why the bilateral shift operator $B$ was introduced in the first place. 

In the study of stationary processes, it appears that the bilateral shift operator $B$ was introduced by \cite{Kolmogorov1941}, and defined as follows assuming a doubly infinite stochastic process $\{X_n\}_{n=-\infty}^{\infty}$.
\begin{gather} \label{eq_unitary_integral}
	B = \int_{z \in S^1} z dE_z
\end{gather}

He does not offer any particular justification for the use of a doubly infinite process as opposed to a singly infinite process, or the bilateral shift operator $B$ as opposed to the unilateral shift operator $T$. \cite[pp.462-463]{Doob1953} goes a little further arguing that for the purposes of statistical parameters like mean and covariance, any given stationary process can be replaced by a doubly infinite Gaussian process with a unitary shift operator, and thus it suffices to study the bilateral shift operator and dispense with the unilateral shift operator. Doob's construction can be regarded as a form of dilation that transforms the underlying random variables, but preserves the statistical properties. In contrast Nagy's dilation keeps the random variables unchanged, and is faithful to the underlying algebra. Doob's argument is valid within careful confines, as long as one restricts oneself to the computation of statistical parameters. But it does not suffice if one extends the analysis to algebraic properties like invertibility. The current article shows that it is more natural to use the unilateral shift operator $T$ in such situations.

\section{Mean Ergodicity} \label{section_mean_ergodicity}

Besides invertibility, the $\ell^1$ condition is often linked to the mean ergodicity of a stationary process (\cite[p.219]{BrockwellDavis1991}, \cite[p.47,52,70]{Hamilton1994}), wherein absolute summability of the moving average coefficients is a sufficient condition for the sample mean to converge to the unbiased true mean. The current section makes it clear that mean ergodicity continues to hold for $f \in H^\infty$ as well.

We first review the mean ergodic theorem, and a working definition of mean ergodicity.
\begin{theorem}[Mean Ergodic Theorem]
If $U$ is an isometry on a Hilbert space, and if $P$ is the projection on the space of all vectors invariant under $U$, then $\frac{1}{n} \sum_{j=0}^{n-1} U^jg$ converges to $Pg$ for every $g$ in the space.
\end{theorem}

If we replace $g$ with $X_t$, we can see that the above sum represents a time average. $Pg$ is often called the space average, and equals the expectation $E[X_t]$. It is clear that the above theorem makes no assumptions about the transfer function $f$.

A stationary process is said to be mean ergodic when the time average of samples equals the space or ensemble average. Mean ergodicity is a property of a stochastic process $X_t$, and it applies to stationary processes with continuous spectral measure where there is no linearly deterministic component. In other words, whenever the time series is described by a unitary operator as $X_t = UX_{t-1}$. 

Mean ergodicity for a transfer function $f \in H^{\infty}$ refers to the mean ergodicity of the associated stationary process $X_t = f(B) \epsilon_t$. The mean ergodicity of $X_t$ follows from the above stated mean ergodic theorem based on the unitary operator $B$, the bilateral shift operator. The statement or proof of the theorem make no reference to any $\ell^1$ condition on the coefficients of $f$, and so no further proof is required that $X_t = f(B) \epsilon_t$ is mean ergodic for $f \in H^{\infty}$. Given the prior (somewhat mistaken) connection to the $\ell^1$ condition, we note some flaws in the literature connecting these ideas, in particular the arguments in \cite[p.47]{Hamilton1994} which we review at this point. Hamilton's example is as follows:
\begin{gather*} \label{hamilton_process_eqn}
X_t^{(i)} = \mu^{(i)} + \epsilon_t
\end{gather*}

With $i$ fixed, $X_t^{(i)}$ is a stationary process. $\mu^{(i)}$ is drawn a normal distribution $N(0, \lambda^2)$ and $\epsilon_t$ is a Gaussian white noise process with mean zero and variance $\sigma^2$. Hamilton claims to prove that mean ergodicity does not hold for $X_t^{(i)}$ since $\mu^{(i)}$ may not be zero. While it is true that time average does not equal ensemble average for $X_t^{(i)}$ the subsequent reasoning is incomplete at best, and perhaps misleading. The flaw is best understood based on the following example by Brockwell and Davis.

\cite[p.190]{BrockwellDavis1991} looks at a stationary process $X_t = Z + \epsilon_t$ where $X_t$ is a stochastic process, $Z$ is a random variable fixed across time and $\epsilon_t$ are white noise processes. In such a setup, $Z$ refers to the deterministic part of the Wold decomposition (since it can be readily predicted), and leads to a discontinuity in the spectral measure. The random variable $Z$ corresponds to $\mu^{(i)}$ in Hamilton's example. The mean and variance of $Z$ cannot be determined based on a time average of $X_t$ since $Z$ is unchanging across multiple samples of $X_t^{(i)}$ for fixed $i$. This is the essential problem. In particular, Hamilton's example fails to be non-ergodic because of the deterministic component of the process $\mu^{(i)}$, not because of a failed $\ell^1$ condition which applies to the non-deterministic part.

Hamilton's commentary tries to link the $\ell^1$ condition with the mean ergodicity property. But the correct reason is that the time average of a stationary process $X_t$ equals the value of the spectral measure at zero, or equivalently, the mean value of the process. In Hamilton's example, there is a discontinuity in the spectral measure at zero, and it represents the deterministic part of the process $Z = \mu^{(i)}$.

To clearly show the spectral discontinuity $\Delta \mu(0)$ referred to above, and its relation to time averages, we include the following result \cite[p.42]{Grenander1957}:
\begin{theorem} [Grenander-Rosenblatt]
	\begin{gather*}
		\frac{1}{N} \sum (X_t^{(i)} - E[X_t]) = \Delta \mu(0) + \frac{1}{N} \int e^{i \lambda} \frac{1 - e^{iN\lambda}}{1 - e^{i\lambda}} d\mu_0(\lambda) \\
		\mu_0(\lambda) = \mu(\lambda),~~\lambda < 0 \\
		\mu_0(\lambda) = \mu(\lambda) - \Delta \mu(0),~~ \lambda \ge 0.
	\end{gather*}
\end{theorem}

We recommend \cite[p.42]{Grenander1957} for a rigorous treatment of mean ergodicity and other topics in stationary time series analysis. The treatments by Hannan and Rozanov seem to be based on this book as well.

\section{Conclusion}

We identified a gap in time series literature in the use of operator valued transfer functions $f(B)$, and addressed the gap with rigorous proofs of existence for the power series expressions denoted by $f(B)$ and $f(T)$, for $f \in H^{\infty}$. 

We applied this rigorous operator calculus to unify the notion of stationary process invertibility based on the roots of the transfer function $f(z)$ with the notion of operator invertibility of $f(T)$. The article, in effect, rationalizes the vocabulary of time series analysis with known results in rigorous functional analysis. We also showed that $f(T)$ represents an isometry, and that $f(T) = T_f$, the Toeplitz operator. 

In terms of future work, we would like to examine the unbounded operator $f(B)$ with $f \in H^2$.

\begin{acknowledgement}
We would like to thank Prof. Nithin Nagaraj for ongoing support, and asking about invertibility; Prof. E.K. Narayanan, Dr. Ajay Shenoy and Dr. Srikanth Pai for helpful and interesting discussions.
\end{acknowledgement}

\bibliography{main}

\end{document}